\documentclass[12pt]{article}
\usepackage[english]{babel}
\usepackage[utf8]{inputenc}
\usepackage{amsthm,amsmath,amsfonts,amssymb}
\usepackage[normalem]{ulem}
\usepackage{nicefrac}
\usepackage{url}
\usepackage{mathtools}
\usepackage{color}
\usepackage{anysize}
\usepackage{geometry}
\usepackage{comment}
\usepackage{dsfont}
\marginsize{2cm}{2cm}{1,5cm}{1,5cm}

\numberwithin{equation}{section}
\theoremstyle{plain}

\newtheorem{theorem}{Theorem}[section]
\newtheorem{lemma}[theorem]{Lemma}
\newtheorem{proposition}[theorem]{Proposition}
\newtheorem{remark}[theorem]{Remark}
\newtheorem{corollary}[theorem]{Corollary}
\newtheorem{definition}[theorem]{Definition}

\usepackage{amsmath}
\usepackage{graphicx}
\usepackage{booktabs}
\definecolor{myblue}{RGB}{20,60,160}
\definecolor{myred}{RGB}{200,16,40}
\usepackage[colorlinks = true,
            linkcolor = myblue,
            citecolor = myred 
            ]{hyperref}
\newcommand{\1}{\mathds{1}}
\mathtoolsset{showonlyrefs}

\title{On the degree distribution of large tentacular Bienaymé--Galton--Watson trees}

\begin{document}
\author{Vanessa Dan\footnote{CMAP, École polytechnique, Institut Polytechnique de Paris, 91120 Palaiseau, France. \\ DMA, École Normale Supérieure, Université PSL, 75005 Paris, France. \\ vanessa.dan@polytechnique.edu.}}
\maketitle
\begin{abstract}
We study Bienaymé--Galton--Watson trees conditioned to have $n$ vertices and $k_n$ leaves, where $k_n\to\infty$ while remaining negligible compared to $n$, i.e. $k_n=o(n)$. More precisely, we determine the asymptotic distribution of the outdegrees of these trees. We first show that the tree is asymptotically binary: the number of vertices with two children is asymptotically equal to the number of leaves, while almost all remaining vertices have exactly one child. Then, we identify the scales at which vertices with larger outdegrees emerge. For every $d\ge2$, the critical scale $k_n \sim cn^{(d-1)/d}$, $c>0$, is the threshold for
the appearance of vertices with outdegree $d+1$. Below this scale, such vertices are absent with high probability; at the critical scale, their number converges to a Poisson distribution; above it, they satisfy a law of large
numbers. Our proofs rely on the coding of Bienaymé--Galton--Watson trees by their Łukasiewicz paths and asymptotic estimates for associated random walks.
\end{abstract}

\section{Introduction}
This article studies the asymptotic behavior of the outdegrees in Bienaymé–Galton–Watson trees (BGW trees) conditioned to have both $n$ vertices and a sublinear number $k_n = o(n)$ of leaves with $k_n\rightarrow \infty$.

Throughout the paper, $\mu$ denotes an offspring distribution on $\mathbb Z_{\geq 0}$ such that $\mu(0)>0$ and $\mu(1)>0$, since otherwise the number of leaves cannot satisfy $k_n=o(n)$, and $T_n^{k_n}$ denotes a $\mu$-BGW tree conditioned to have $n$ vertices and $k_n$ leaves. For every $i\geq 0$, we let $\phi_i(t)$ be the number of vertices of a tree $t$ having exactly $i$ children.
\begin{theorem}\label{Th_size_R} 
Assume that $\mu(2)>0$. When $k_n=o(n)$ and $k_n\rightarrow \infty$, we have with probability tending to $1$
 \[\phi_2\left(T_n^{k_n}\right) \underset{n \to \infty}{\sim} k_n  \qquad \text{and} \qquad \phi_1\left(T_n^{k_n}\right) = n-2k_n +o(k_n). \] 
\end{theorem}
 
\noindent This result shows that the tree behaves similarly to a binary tree with elongated branches: it has a number of branching vertices (those that have two or more children) asymptotically equivalent to the number of leaves and among all branching vertices, the proportion of those with three or more children tends to 0. A natural question is then to understand how vertices with larger outdegrees appear as the number of leaves increases. For any tree $t$, we denote by $\deg_{\max}(t)$ its maximal outdegree. We write $\mathrm{Poi}(\lambda)$ the Poisson distribution with parameter $\lambda$. 
\begin{theorem}\label{Th_degrees}
    Let $d \in \mathbb{Z}_{\geq 2}$. Assume that $k_n = o(n)$, $k_n\rightarrow \infty$ and $\mu(2)>0$.
    \begin{enumerate}
        \item If $k_n=o( n^{\frac{d-1}{d}})$, then \[\mathbb{P}\left(\mathrm{deg_{max}}(T_n^{k_n}) \leq d  \right)\xrightarrow[n\to\infty]{}1.\]
        \item Let $c>0$. If $k_n\underset{n \to \infty}{\sim} cn^{\frac{d-1}{d}}$, 
        \[\phi_{d+1}\left(T_n^{k_n}\right) \xrightarrow[n\to\infty]{(d)} \mathrm{Poi}\left( c^d\frac{\mu(d+1)\mu(1)^{d-1}}{\mu(2)^d} \right). \]
        \item If $k_n^d/n^{d-1} \rightarrow + \infty$ as $n\rightarrow +\infty$, then for every $p \geq 1$ 
\[ \frac{n^{d-1}}{k_n^d}\phi_{d+1}\left(T_{n}^{k_n}\right) \xrightarrow[n\to\infty]{L^p} \frac{\mu(1)^{d-1}\mu(d+1)}{\mu(2)^d}.\]
\end{enumerate}
\end{theorem}

\noindent Assume that $k_n \!\underset{n \to \infty}{\sim} \! cn^{\frac{d-1}{d}}$ for some $d \in \mathbb{Z}_{\geq 2}$ and $c>0$. Applying the first item of Theorem~\ref{Th_degrees} with $d+1$, we have that vertices with outdegree larger than $d+1$ are absent with high probability. Consequently, the maximal outdegree is asymptotically at most $d+1$. Moreover, for every $i\in [\![3,d]\!]$, the number of vertices with outdegree $i$ grows polynomially with $n$ and satisfies a law of large numbers. More precisely, for every $p\geq 1$,
\[\frac{\phi_i(T_n^{k_n})}{n^{1-\frac{i-1}{d}}}
    \xrightarrow[n\to\infty]{L^p}
    c^{i-1}\frac{\mu(1)^{i-2}\mu(i)}{\mu(2)^{i-1}}. \]
In summary, Theorem~\ref{Th_degrees} shows that the scale $n^{(d-1)/d}$ plays a special role for the appearance of vertices with outdegree $d+1$. Below this scale such vertices are absent with high probability, at this scale their number converges to a Poisson distribution, and above it they become abundant and satisfy a law of large numbers. Note that Theorem~\ref{Th_degrees} could be extended to cases where $\mu(2)=0$ by arguing as in \cite{DAN25} (see the proof of Theorem 3.4), i.e. by introducing a suitable class of admissible $\mu$-BGW trees with $k_n$ leaves.\\ 

\noindent Let us illustrate Theorem~\ref{Th_degrees} with the cases where $\mu$ is geometric or Poisson, which are of particular interest since the corresponding conditioned trees are respectively uniform plane trees and uniform Cayley trees with $n$ vertices and $k_n$ leaves.
If $\mu$ is the geometric distribution on $\mathbb{Z}_{\geq0}$ with parameter $p \in (0,1)$, the quantity $\mu(d+1)\mu(1)^{d-1}/\mu(2)^d$ is equal to $1$, so Theorem \ref{Th_degrees} implies that for every $d \in \mathbb{Z}_{\geq 2}$, if $k_n \underset{n \to \infty}{\sim} cn^{(d-1)/d}$ then
\[ \phi_{d+1}(T_n^{k_n})
\xrightarrow[n\to\infty]{(d)}
\mathrm{Poi}(c^d)\]
and for every $p\geq 1$, if $k_n^d/n^{d-1} \rightarrow + \infty$ then
\[\frac{n^{d-1}}{k_n^d}\phi_{d+1}(T_n^{k_n})
\xrightarrow[n\to\infty]{L^p}
1.\]
Similarly, when $\mu$ is the Poisson distribution with parameter $\lambda >0$, we have $\mu(d+1)\mu(1)^{d-1}/\mu(2)^d$ is equal to $2^d/(d+1)!$. Hence, Theorem~\ref{Th_degrees} yields for every $d \in \mathbb{Z}_{\geq 2}$ that if $ k_n \underset{n \to \infty}{\sim} cn^{(d-1)/d}$,
\[ \phi_{d+1}(T_n^{k_n})
\xrightarrow[n\to\infty]{(d)}
\mathrm{Poi}\left(\frac{2^d}{(d+1)!}c^d\right) \]
and for every $p\geq 1$, if $k_n^d/n^{d-1} \rightarrow + \infty$ then
\[ \frac{n^{d-1}}{k_n^d}\phi_{d+1}(T_n^{k_n})
\xrightarrow[n\to\infty]{L^p}
\frac{2^d}{(d+1)!}.\]

\paragraph{Motivations.} The total number of vertices with fixed outdegrees in BGW conditioned on their total progeny only was first studied by Kolchin \cite{K86}, who showed that, in large conditioned critical BGW trees with finite variance, the number of vertices having a prescribed outdegree is asymptotically normal. Minami \cite{MN05} established that these convergences hold jointly under additional moment condition, which was later lifted by Janson \cite{J16}.
Finally, Thévenin \cite{T20} analyzed the evolution of the number of vertices whose outdegree belongs to a fixed set $A$ during various exploration procedures of a tree conditioned to have $n$ vertices. Moreover, for  trees conditioned to have $n$ vertices whose outdegree belongs to a fixed set $B$, he established a multidimensional central limit theorem for the numbers of vertices whose outdegrees belong to the sets $A_1,\ldots,A_k$.

Recently, several works have studied biconditioned trees, i.e. trees conditioned on both their total number of vertices and their number of leaves (see \cite{LM07, Kargin23} for scaling limits and \cite{ABD23} for local limits). Motivated by questions related to random maps, Kortchemski and Marzouk obtained results \cite{KM23} on the scaling limits of the Łukasiewicz walk of such trees. 
Finally, in \cite{DAN25}, the author studied BGW trees conditioned to have a large number of vertices and either a fixed number of leaves or a fixed number of internal nodes. The global geometry of the limiting tree was characterized, which in particular provides a complete understanding of the behavior of its outdegrees.  
The purpose of the present article is to extend the results of \cite{DAN25} from the regime of a fixed number of leaves to the more general regime where the number of leaves is allowed to grow with $n$, while remaining negligible compared to the total size of the tree, i.e. $k_n = o(n)$.

\paragraph{Techniques.} We now comment on the main tools involved in the proofs. Our approach relies heavily on the classical coding of Bienaymé--Galton--Watson trees by their \L ukasiewicz paths and the cyclic lemma. Through this coding, questions about the outdegrees of a biconditioned BGW tree can be translated into questions about suitable random walks with nonnegative increments. As a consequence, many quantities of interest can be expressed in terms of probabilities of the form $\mathbb P(S_n=x_n)$, where $(S_n)$ is an unconditioned random walk and $(x_n)$ is a suitable sequence. The asymptotic analysis of such probabilities relies on a result of Kortchemski and Marzouk \cite{KM23}. Although explicit asymptotic formulas are not always available, the quantities arising in our arguments often appear through ratios of such terms, which turn out to be enough to derive the desired asymptotic estimates.

\paragraph{Structure of the paper.} The paper is organized as follows. Section \ref{sec2} recalls basic definitions and properties of Bienaymé–Galton–Watson trees, as well as the coding function used to analyze their structure. Section \ref{sec3} introduces the biconditioning on the total number of vertices and leaves, and contains the proof of Theorem \ref{Th_size_R}. Finally, Section \ref{sec4} is devoted to the proof of Theorem \ref{Th_degrees}.

\tableofcontents

\paragraph{Acknowledgements.} I would like to sincerely thank my PhD advisors, Igor Kortchemski and Cyril Marzouk, for their invaluable advice, support, and feedback throughout this work.

\section{Background on rooted planar trees and BGW trees} \label{sec2}
\begin{table}[htbp]\caption{Table of the main notation introduced in Sections \ref{sec2}, \ref{sec3} and used later.}
\centering
\begin{tabular}{c c p{12cm} }
\toprule
$\mathbb{T}_{i}$ && the set of trees with $i$ vertices\\
$\mathbb{T}^j_{i}$ && the set of trees with $i$ vertices and $j$ leaves  \\
$c_u(a)$ && the number of children of the vertex $u$ in the tree $a$ \\
$\phi_i(a)$ && $|\{u \in a \mid c_u(a) = i \}|$ for a tree $a$ with $n$ vertices\\
$T,R$ && a $\mu$-BGW tree and its reduced tree  \\
$T^k_{n},R^k_{n}$ && a $\mu$-BGW tree conditioned to have $n$ vertices and $k$ leaves and  $R(T^k_{n})$ \\
\hline
$\nu(i)$  && $ \mu(i+1)$, $\forall i \geq -1 $ \\
$\xi(i)$ && $\mu(i+1)/(1-\mu(0))$, $\forall i \geq 0 $ \\
$(X_i)_{i\geq0}$ && a sequence of i.i.d. random variables with distribution $\nu$ \\
$(Y_i)_{i\geq0}$ && a sequence of i.i.d. random variables such that $\mathbb{P}(Y_1=-1) = 0$ and $\forall i \geq 0,$ $\mathbb{P}(Y_1=i) = \xi(i)$ \\
$W_n$ && $ X_1 +\cdots+X_n$\\
$W'_n$ && $ Y_1 +\cdots+Y_n$\\
\bottomrule
\end{tabular}
\end{table}

In this section, we recall basic definitions and properties of trees and random trees that will be used throughout the paper. For further details and proofs of the stated results, the reader is referred to \cite{LG05, PC06}.

\subsection{Rooted planar trees and \L ukasiewicz paths}
Let $\mathcal{U} \coloneqq \bigcup_{n=0}^\infty (\mathbb{Z}_{>0})^n $ denote the set of words on the alphabet $\mathbb{Z}_{>0}$, where $\mathbb{Z}_{>0}^0 = \{\varnothing\} $. If $u,v \in \mathcal{U},$ then $ uv$ denotes the concatenation of $u$ and $v$. If $u,u',v \in \mathcal{U}$ and $u=vu'$, then $v$ is said to be an ancestor of $u$. In particular if $u' \in \mathbb{Z}_{>0}$, $v$ is said to be the parent of $u$ and $u$ a child of $v$. 
\begin{definition}
    A rooted planar tree $\tau$ is a subset of $\mathcal{U}$ such that 
    \begin{enumerate}
        \item $\varnothing \in \tau$,
        \item if $u = vu' \in \tau$ then $v \in \tau$,
        \item for all $u \in \tau$, there exists $c_u(\tau) \in \mathbb{Z}_{\geq 0}$ such that $uj \in \tau$ if and only if $j \leq c_u(\tau)$; $c_u(\tau)$ is called the number of children of $u$ in $\tau$.
    \end{enumerate}
\end{definition}
\noindent See Figure \ref{Fig_def_tree_and_luka} for an example. For every rooted planar tree $\tau$, $\varnothing$ is called the root of $\tau$, every $u \in \tau$ is a vertex (or a node) of $\tau$. If $u$ has no child it is a leaf of $\tau$ and it is an internal node otherwise. Finally, $|\tau|$ denotes the number of vertices in $\tau$. One can easily check that for every tree $\tau$,
$\sum_{u\in \tau}c_u(\tau)=|\tau|-1$. In the following, we will only consider rooted planar trees, so from now on we refer to them as `trees' for simplicity. 
Let us denote by $\mathbb{T}$ the set of such trees and $\mathbb{T}_n$ the set of trees having $n$ vertices. \vspace{0.2cm}

\noindent At times we will explore a tree, meaning that we visit its vertices one by one. It is then necessary to choose the order in which the vertices are visited. A natural order on trees is given by the lexicographic order on the alphabet $\mathbb{Z}_{>0}\cup\{\varnothing\}$. 
Now, let us define a coding function of trees. 
\begin{definition} Let $\tau$ be a finite tree and $u_0,u_1,...,u_{|\tau|-1}$ be its vertices listed in the lexicographic order. The \L ukasiewicz path of $\tau$, denoted by $W(\tau)=(W_i(\tau), 0 \leq i \leq |\tau|)$ is defined as follows:
\[ W_i(\tau) = \left\{
    \begin{array}{ll}
        0 & \mbox{for } n=0 \\
        W_{i-1}(\tau) + c_{u_{i-1}}(\tau) - 1 & \mbox{for all } 1\leq i \leq |\tau|.
    \end{array}
\right. \]
\end{definition}
See Figure \ref{Fig_def_tree_and_luka} for an example.
\begin{figure}[h]
\centering
        \includegraphics[scale=1]{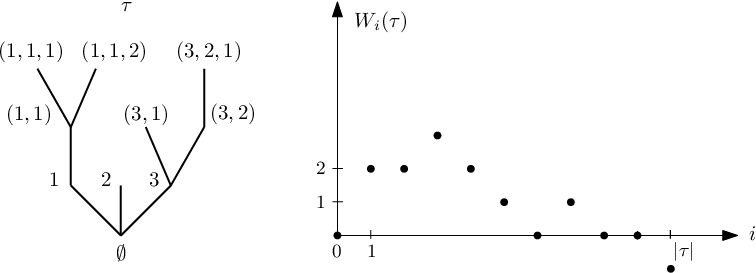}
        \caption{An example of a rooted planar tree $\tau$ and its \L ukasiewicz path.}
        \label{Fig_def_tree_and_luka}
\end{figure}

\noindent The following well-known result establishes a bijection between finite trees and paths.
\begin{proposition} \label{bij_arbre_luka}
  Let $n \in \mathbb{N}$. The map $\tau \mapsto W(\tau)$ defines a bijection between $\mathbb{T}_n$ and the set of paths $f:[\![0,n]\!] \rightarrow \mathbb{Z}$ such that \[ f(0)=0, \quad f(n)=-1, \quad f(i) \geq 0 \text{ and } f(i+1)-f(i) \in \mathbb{Z}_{\geq -1}  \text{ for } i \in [\![0,n-1]\!] .\]
\end{proposition}

\subsection{BGW trees and random walks}

From now on, we focus on a particular family of random trees,  Bienaymé–Galton–Watson trees. Roughly speaking, given a probability $\mu$ on $\mathbb{Z}_{\geq0}$, a  Bienaymé–Galton–Watson tree with offspring distribution $\mu$ is a random tree coding the genealogy of a population starting with one individual and where all individuals reproduce independently of each other according to the distribution $\mu$. Let us give a proper definition. 
\begin{definition}
    Let $\mu$ be an offspring distribution on $\mathbb{Z}_{\geq 0}$ and $(C_{u})_{u\in \mathcal{U}}$ be independent and indentically distributed random variables of law $\mu$. A random tree is a  Bienaymé–Galton–Watson tree with offspring distribution $\mu$ ($\mu$-BGW tree) if it has the same law as the random tree $T$ defined by \[T=\left\{u=(u_1,\dots,u_m)\in\mathcal U:
u_j\le C_{(u_1,\dots,u_{j-1})}
\text{ for every }1\le j\le m\right\}.\]
\end{definition}

In what follows, $n$ and $k_n$ are integers satisfying $k_n=o(n)$. We denote by $T$ a $\mu$-BGW tree and by $T_{n}$ such a tree conditioned on having $n$ vertices. One can note that the probability of $T$ being equal to a finite tree $\tau$ is explicit: 
\[\mathbb{P}(T=\tau)= \prod_{ u \in \tau}\mu(c_u(\tau)).\]
We define $W=(W_i, 0 \leq i \leq n)$ as the random walk started from $W_0=0$ with step-distribution $\nu$, where $\nu(i) =\mu(i+1) $ for all $ i \geq -1$. Let $H_{k}(W)\coloneqq\inf\{i\geq 0 : W_i=k\}$ denote the first hitting time of $k \in \mathbb{Z}$ by the walk $W$.

\begin{proposition}\label{propcodingluka}
 The \L ukasiewicz path $W(T_n)=(W_i(T_n), 0 \leq i \leq n)$ associated with $T_n$ has the same distribution as $W=(W_i, 0 \leq i \leq n)$ conditioned on $H_{-1}(W)=n$.
 \end{proposition}
\noindent  Conditioning on $H_{-1}(W) = n$ can be delicate to handle. Proposition \ref{consequencecyclelemma}, a consequence of the cycle lemma, provides a practical alternative.

\begin{proposition}\label{consequencecyclelemma} 
   Let $X_1,\ldots,X_n$ denote the i.i.d. increments of $W=(W_i, 0 \leq i \leq n)$. Let $f$ be a measurable, positive and stable by cyclic permutation function. Then, we have the following equality:
   \[\mathbb{E}\left[ f \left(X_1,\ldots,X_n\right) \1_{W_n=-1, W_i \geq 0 \, \forall i \in [\![1,n-1]\!]} \right] = \frac{1}{n} \mathbb{E}\left[ f \left(X_1,\ldots,X_n\right) \1_{W_n=-1} \right]. \]
\end{proposition}

We now present some results on BGW trees and random walks that will be used several times in the sequel.
We define the probability distribution $\xi$ on $\mathbb{Z}_{\geq 0}$ by \[\xi \coloneqq \frac{\mu(\cdot+1)}{1-\mu(0)}.\]
Let $W'$ be the random walk started from $W'_0=0$ with step distribution $\xi$. In other words, $W'$ has the same distribution as $W$ conditioned to have only nonnegative increments.

\begin{proposition}\label{classiquedecomposition}
     Let $(X_i)_{i \geq 1}$ and $(Y_i)_{i \geq 1}$ be the i.i.d. increments of $W$ and $W'$, respectively. We define $(\widehat{X}_i)_{i \geq 1}$ as the sequence obtained from $(X_i)_{i \geq 1}$ by removing all steps equal to $-1$. 
     Let $f : \mathbb{R}^{n-k_n} \to \mathbb{R}_+ $ be a measurable and stable by cyclic permutation function.
     Then, we have the following equality:
     \begin{align*}
   & \mathbb{E}\left[ f\left( (\widehat{X}_i)_{1\leq i \leq n-k_n}\right) \1_{ W_n=-1, \ W_i \geq 0 \, \forall i \in [\![1,n-1]\!], \ |\{i \in [\![1,n]\!] : X_{i}=-1\}|=k_n  }\right] \\ & \qquad = \frac{1}{n}\binom{n}{k_n}\mu(0)^{k_n}\left(1-\mu(0) \right)^{n-k_n}  \mathbb{E}\left[ f\left( (Y_i)_{1\leq i \leq n-k_n}\right) \1_{W'_{n-k_n}=k_n-1} \right].
   \end{align*}
\end{proposition}
\begin{remark}
By the coding of a tree via its \L ukasiewicz path (Proposition~\ref{propcodingluka}), probabilities of the form appearing in Proposition~\ref{classiquedecomposition} arise naturally. The key point is that conditioning on the number of the leaves reduces the problem to a random walk with nonnegative increments.
\end{remark}
\begin{proof}[Proof of Proposition \ref{classiquedecomposition}]
By applying Proposition \ref{consequencecyclelemma}, we have
\begin{align*}
   & \mathbb{E}\left[ f\left( (\widehat{X}_i)_{1\leq i \leq n-k_n}\right) \1_{ W_n=-1, \ W_i \geq 0 \, \forall i \in [\![1,n-1]\!], \ |\{i \in [\![1,n]\!] : X_{i}=-1\}|=k_n  }\right] \\
   & \qquad = \frac{1}{n} \mathbb{E}\left[f\left( (\widehat{X}_i)_{1\leq i \leq n-k_n}\right) \1_{W_n=-1,\ |\{i \in [\![1,n]\!]: X_i=-1\}|=k_n,} \right].
   \end{align*}
We then partition according to the choice of the $k_n$ indices at which $X_i=-1$, and we obtain that the above expectation equals \[\frac{1}{n}\sum_{\substack{ I \subset [\![1,n]\!]\\
              |I|=k_n}}  \mathbb{E}\left[f\left( (X_i)_{i \notin I} \right)  \1_{X_i \neq -1 \, \forall i \notin I, \ \sum_{j \notin I} X_j = k_n-1 } \ \1_{X_i = - 1 \, \forall i \in I } \right].\]
By independence, this quantity is equal to 
\begin{equation} \label{equ_preuve_decompo}
      \frac{1}{n}\sum_{\substack{ I \subset [\![1,n]\!]\\
              |I|=k_n}}  \mathbb{P}\left( X_i = -1 \, \forall i \in I \right) \mathbb{E}\left[f\left( (X_i)_{i \notin I} \right)  \1_{X_i \neq -1 \, \forall i \notin I, \ \sum_{j \notin I} X_j = k_n-1 }\right].
\end{equation}
For any $ I \subset [\![1,n]\!]$ of size $k_n$, we have \[\mathbb{P}\left( X_i = -1 \, \forall i \in I \right) = \mu(0)^{k_n}\] and by definition of $(Y_i)_{i\geq 1}$ and $\xi$,  \[\mathbb{E}\left[f\left( (X_i)_{i \notin I} \right)  \1_{X_i \neq -1 \, \forall i \notin I, \ \sum_{j \notin I} X_j = k_n-1 }\right] = \left( 1-\mu(0)\right)^{n-k_n} \mathbb{E}\left[f\left( (Y_i)_{1\leq i \leq n-k_n} \right)  \1_{\ \sum_{j=1}^{n-k_n} Y_j = k_n-1 }\right]. \]
Note that the right hand size does not depend on $I$ so  by definition of  $W'$, we finally get that the probability \eqref{equ_preuve_decompo} is equal to 
\[ \frac{1}{n}\binom{n}{k_n} \mu(0)^{k_n} (1-\mu(0))^{n-k_n} \mathbb{E}\left[f\left( (Y_i)_{1\leq i \leq n-k_n} \right)  \1_{ W'_{n-k_n} = k_n-1 }\right],\]
which concludes the proof.
\end{proof}
 \noindent We now state an immediate corollary of Proposition \ref{classiquedecomposition} corresponding to the particular case when $f$ is the constant function $1$. 
\begin{corollary} \label{propTW} 
Let $\mu$ be an offspring distribution and $T_n^{k_n}$ a $\mu$-BGW tree conditioned to have $n$ vertices and $k_n$ leaves. Then, we have the following equality
\begin{equation*}
 \mathbb{P}\left(T \in \mathbb{T}^{k_n}_{n}\right) = \frac{1}{n}\mathbb{P}\left(B_{n}=k_n\right)\mathbb{P}\left(W'_{n-k_n} = k_n-1\right),   
\end{equation*}
where $B_{n}$ is the sum of $n$ independent Bernoulli random variables with parameter $\mu(0)$ and  $W'$ the random walk started from $W'_0=0$ with step distribution $\xi$.
\end{corollary}
Finally, we recall Theorem 1.1.(ii) from \cite{KM23}, applied with the choices $n=n-k_n$, $x_n=k_n$, $v_n= \sqrt{k_n}$. This result will play a key role in what follows.
\begin{theorem} \label{thKM}
Let $\eta$ be an offspring distribution such that $\eta(0) >0$ and $\eta(1) >0$ and denote by $F$ its generating function. Let $(S_n)_{n\geq0}$ be a random walk started from $S_0=0$ with step distribution $\eta$. Assume $k_n=o(n)$ and $k_n\rightarrow \infty$. Then
    \begin{equation}\label{eqKMf1}
    \sup_{k\geq -k_n} \left| \sqrt{k_n} \frac{b_n^{k_n+k}}{F(b_n)^{n-k_n}}\mathbb{P}\left(S_{n-k_n} = k_n+k\right) - \frac{1}{\sqrt{2\pi}}e^{-\frac{k^2}{2k_n}}\right| \xrightarrow[n \rightarrow \infty]{} 0,
\end{equation}
where $b_n$ is defined by
 \[b_n\frac{F'(b_n)}{F(b_n)}=\frac{k_n}{n-k_n}.\]
\end{theorem}
\noindent We also derive some useful consequences.
\begin{corollary} \label{corKM}
We keep the notation and assumptions of Theorem \ref{thKM}. Let $r \geq 0$ be a fixed integer.
\begin{enumerate}
\item[(i)] \label{corKM1}For $n$ large enough let $b_{n}^{(r)}$ be defined by
$$ b^{(r)}_{n} \frac{F'(b^{(r)}_{n})}{F(b^{(r)}_{n})}= \frac{k_{n}}{n-k_n-r}.$$ We have
$$b_{n} \underset{n \to \infty}{\sim} \frac{k_{n}}{n} \frac{\mathbb{P}(\eta=0)}{\mathbb{P}(\eta=1)} \qquad\textrm{and} \qquad \frac{b^{(r)}_{n}}{b_{n}}=1 + O \left(\frac{1}{n}\right).$$
\item[(ii)] \label{corKM2} There exists $C>0$ such that, for $n$ large enough and uniformly for $-k_n\le k\le0$,
$$
\mathbb{P}(S_{n-k_n-r}=k_n+k)\le \frac{C}{\sqrt{k_n}}\frac{F(b_n)^{n-k_n-r}}{b_n^{k_n+k}}.
$$
\item[(iii)] \label{corKM3} For every fixed $s \in \mathbb{Z}$ we have
$$
\mathbb{P}(S_{n-k_n-r}=k_n-s)  \underset{n \to \infty}{\sim}  \frac{1}{\sqrt{2\pi k_n}}\frac{F(b_n)^{n-k_n-r}}{b_n^{k_n-s}}.
$$
\end{enumerate}
\end{corollary}

\begin{proof}
Set $h(x)\coloneqq xF'(x)/F(x)$. The function $h$ is nondecreasing on $[0,1)$; indeed, if $\eta_x$ denotes the tilted law $\mathbb{P}(\eta_x=j)=\eta(j)x^j/F(x)$, then $xh'(x)=\operatorname{Var}(\eta_x)\ge0$. Since $\eta(0)>0$ and $\eta(1)>0$, we have $F(x)\to \mathbb{P}(\eta=0)$ and $F'(x)\to \mathbb{P}(\eta=1)$ as $x\downarrow0$, and therefore $h(x)\sim x\mathbb{P}(\eta=1)/\mathbb{P}(\eta=0)$. As $h(b_n)=k_n/(n-k_n)\to0$, this implies $b_n\to0$, and hence
$$
b_n=\frac{k_n}{n-k_n}\frac{F(b_n)}{F'(b_n)} \underset{n \to \infty}{\sim}  \frac{k_n}{n}\frac{\mathbb{P}(\eta=0)}{\mathbb{P}(\eta=1)}.
$$
Similarly, since $r$ is fixed, $h(b_n^{(r)})=k_n/(n-k_n-r)\to0$, so $b_n^{(r)}\to0$. Moreover, $h'(0)=\mathbb{P}(\eta=1)/\mathbb{P}(\eta=0)>0$, so $h'$ is bounded away from $0$ in a neighborhood of $0$. By the mean value theorem applied with $h$,
$$
|b_n^{(r)}-b_n|\le C\left|\frac{k_n}{n-k_n-r}-\frac{k_n}{n-k_n}\right|=O\left(\frac{k_n}{n^2}\right).
$$
Dividing by $b_n\sim k_n\mathbb{P}(\eta=0)/(n\mathbb{P}(\eta=1))$ gives $b_n^{(r)}/b_n=1+O(1/n)$, proving $(i)$.

We now prove $(ii)$. Applying Theorem \ref{thKM} with $n$ replaced by $n-r$, whose associated tilting parameter is $b_n^{(r)}$, we get, for $n$ large enough and uniformly for $k\ge -k_n$,
\[\mathbb{P}\left(S_{n-k_n-r}=k_n+k\right) \leq \frac{C}{\sqrt{k_n}}\frac{F\big(b_n^{(r)}\big)^{n-k_n-r}}{\big(b_n^{(r)}\big)^{k_n+k}}.\]
It remains to replace $b_n^{(r)}$ by $b_n$. Since $b_n^{(r)}-b_n=O(k_n/n^2)$ and $F(0)>0$, we have $\log(F(b_n^{(r)})/F(b_n))=O(k_n/n^2)$, while $b_n^{(r)}/b_n=1+O(1/n)$ gives $\log(b_n/b_n^{(r)})=O(1/n)$. Therefore, uniformly for $-k_n\le k\le0$,
$$
(n-k_n-r)\log\frac{F\big(b_n^{(r)}\big)}{F(b_n)}+(k_n+k)\log\frac{b_n}{b_n^{(r)}}=O\left(\frac{k_n}{n}\right)=o(1).
$$
Hence
$$
\frac{F\big(b_n^{(r)}\big)^{n-k_n-r}}{\big(b_n^{(r)}\big)^{k_n+k}}\le C'\frac{F(b_n)^{n-k_n-r}}{b_n^{k_n+k}},
$$
uniformly for $-k_n\le k\le0$. This proves $(ii)$.

Finally, let $s\in\mathbb{Z}$ be fixed. For $n$ large enough, $-s\ge -k_n$, and Theorem \ref{thKM} applied again with $n$ replaced by $n-r$ gives
$$
\mathbb{P}(S_{n-k_n-r}=k_n-s) \underset{n \to \infty}{\sim} \frac{1}{\sqrt{2\pi k_n}}\frac{F\big(b_n^{(r)}\big)^{n-k_n-r}}{\big(b_n^{(r)}\big)^{k_n-s}}e^{-s^2/(2k_n)}.
$$
Since $s$ is fixed, $e^{-s^2/(2k_n)}=1+o(1)$. Moreover, the logarithmic estimate above, now with $k=-s$, gives
$$
\frac{F\big(b_n^{(r)}\big)^{n-k_n-r}}{\big(b_n^{(r)}\big)^{k_n-s}} \underset{n \to \infty}{\sim} \frac{F(b_n)^{n-k_n-r}}{b_n^{k_n-s}}.
$$
This proves $(iii)$ and completes the proof.
\end{proof}

\section{Size of $R_n^{k_n}$ and $T_n^{k_n}$} \label{sec3}
 The aim of this section is to prove Theorem \ref{Th_size_R}. In the whole section, we work in the regime $k_n=o(n)$. We begin by introducing some notation and a transformation of trees that will play a central role in the sequel.
\subsection{Definitions}
We denote by $T$ a $\mu$-BGW tree and by $T^{k_n}_{n}$ such a tree conditioned on having $n$ vertices and $k_n$ leaves (when this conditioning is non degenerate). In the sequel, $\mathbb{T}_{n}$ will be the set of trees having $n$ vertices, $\mathbb{T}^{k_n}$ the set of trees with $k_n$ leaves and $\mathbb{T}_{n}^{k_n}$ the set of trees having $n$ vertices and ${k_n}$ leaves. For every tree $a$, recall that $c_u(a)$ denotes the number of children of the vertex $u$ in $a$ and $\phi_i(a)$ the number of vertices having exactly $i$ children, i.e.
\[
\phi_i(a) \coloneqq \left|\{u \in a : c_u(a) = i\} \right|.
\]

\begin{definition} Let $ a \in \mathbb{T}_{n}^{k_n}$.
The reduced tree of $a$, denoted by $R(a)$, is obtained by removing all vertices in $a$ that have exactly one child. See Figure \ref{Fig1.2} for an illustration.
\end{definition}
\begin{figure}[h!]
\centering
        \includegraphics[scale=1]{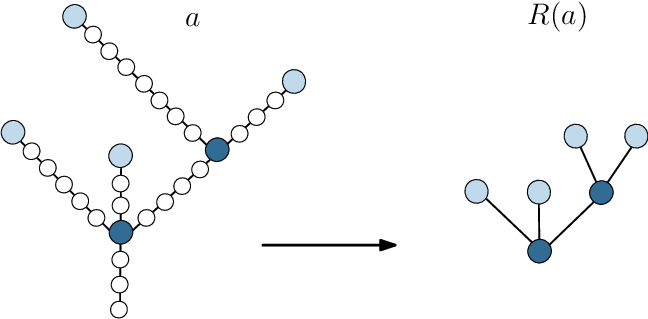}
        \caption{An example of a tree $a$ with its reduced tree $R(a)$.}
\label{Fig1.2}
\end{figure}
\noindent To simplify notation, if $T$ is a $\mu$-BGW tree, we write $R$ for its reduced tree. Similarly, for a conditioned tree $T_n^{k_n}$, we set $R_n^{k_n} \coloneqq R(T_n^{k_n})$.

\subsection{Proof of Theorem \ref{Th_size_R}}
We first establish that, in the regime $k_n = o(n)$ where $k_n\rightarrow \infty$, most internal vertices of $R_n^{k_n}$ have exactly two children. In particular, the size of $R_n^{k_n}$ is asymptotically that of a binary tree with $k_n$ leaves.
\begin{proposition}\label{prop_size_R}
When $k_n=o(n)$, $k_n\rightarrow \infty$ and $\mu(2)>0$, we have with probability tending to $1$:
 \[\phi_2\left(R_n^{k_n}\right) \underset{n \to \infty}{\sim} k_n  \qquad \text{and} \qquad \left|R_n^{k_n}\right| \underset{n \to \infty}{\sim} 2k_n  .\]
\end{proposition}
\noindent Theorem \ref{Th_size_R} is a direct consequence of this proposition since $\phi_2(T_n^{k_n})=\phi_2(R_n^{k_n})$ and $\phi_1(T_n^{k_n}) = n - | R_n^{k_n}|$. The key step is the following lemma, which shows that vertices with outdegrees at least three contribute $o(k_n)$, and therefore form a negligible proportion of the branching vertices.

\begin{lemma}\label{prop_deg_Rnk}
    Assume that  $k_n=o(n)$, $k_n\rightarrow \infty$ and $\mu(2)>0$. Let $\epsilon > 0$. We have the following convergence \[\mathbb{P}\left(\sum_{d\geq2}d\phi_{d+1}\left(R_{n}^{k_n} \right) \geq \epsilon k_n\right) \xrightarrow[n\to\infty]{}0.\]
\end{lemma}
\begin{proof}
      Let $\epsilon > 0$. Applying the Markov inequality and Fubini's theorem, we have the following inequality:
      \begin{align}
\mathbb{P}\left(\sum_{d\geq2}d\phi_{d+1}\left(R_{n}^{k_n} \right) \geq \epsilon k_n\right) & = \mathbb{P}\left(\sum_{d\geq2}d\phi_{d+1}\left(R \right) \geq \epsilon k_n  \mid T \in \mathbb{T}_{n}^{k_n} \right)  \\ & \label{Markov} \leq \frac{1}{\epsilon k_n} \sum_{d\geq2}d \,\mathbb{E}\!\left[ \phi_{d+1}(R) \mid T \in \mathbb{T}_{n}^{k_n} \right].
\end{align}
Using the coding of $T$ by its \L ukasiewicz path (Proposition \ref{propcodingluka}), if $W$
denotes the random walk started from $W_0=0$ and with i.i.d. increments $(X_{i})_{i \in \mathbb{Z}_{\geq 0}}$ distributed according to $\nu$ defined by $ \nu(i) =\mu(i+1) $ for all $ i \geq -1$, we deduce that 
\[\mathbb{E}\left[\phi_{d+1}(R)\1_{T \in \mathbb{T}^{k_n}_n}\right] = \mathbb{E}\left[ \left| \left\{ j\in [\![1,n-k_n]\!] : \widehat{X}_j=d \right\}\right| \1_{ W_n=-1, \ W_i \geq 0 \, \forall i \in [\![1,n-1]\!], \ |\{i \in [\![1,n]\!] : X_{i}=-1\}|=k_n  }\right]. \]
Applying Proposition \ref{classiquedecomposition} with $f(x_1,\ldots,x_{n-k_n}) = \vert\{ j \in [\![1,n-k_n]\!] : x_j=d\}\vert$, we obtain that the latter expectation is equal to
\begin{align*}
 \frac{1}{n} \binom{n}{k_n}\mu(0)^{k_n} (1-\mu(0))^{n-k_n}\mathbb{E}\left[ \left| \left\{ j\in [\![1,n-k_n]\!] : Y_j=d \right\}\right| \1_{W'_{n-k_n}=k_n-1} \right],
\end{align*}
 where $W'$ is a random walk starting at $0$ with i.i.d. increments $(Y_{i})_{i \in \mathbb{Z}_{\geq 0}}$ distributed according to $\xi$. Recall that $\xi$ is the probability distribution on $\mathbb{Z}_{\geq 0}$ defined by \[\xi = \frac{\mu(\cdot+1)}{1-\mu(0)},\]
and let $G$ be the generating function of $\xi$.
Moreover, we have by independence and symmetry \[\mathbb{E}\left[  \left| \left\{ j\in [\![1,n-k_n]\!] : Y_j=d \right\}\right| \1_{W'_{n-k_n}=k_n-1} \right] = (n-k_n)\mathbb{P}\left(W'_{n-k_n}=k_n-1, Y_1 = d \right),\]
and by the definition of $W'$ and $(Y_i)_{1\leq i \leq n-k_n}$, we obtain 
\begin{align*}
  \mathbb{P}\left(W'_{n-k_n}=k_n-1, Y_1 = d \right) & = \mathbb{P}\left(Y_1 = d  \right)\mathbb{P}\left(W'_{n-k_n}=k_n-1 \mid Y_1 = d  \right) \\ & = \frac{\mu(d+1)}{1-\mu(0)}\mathbb{P}\left(W'_{n-k_n-1}=k_n-d-1 \right).
\end{align*} 
We finally deduce that for every $d \in \mathbb{Z}_{\geq 2}$ 
\begin{align}\label{espcond}
    \mathbb{E} \left[ \phi_{d+1}(R) \1_{T \in \mathbb{T}_{n}^{k_n}} \right] \! = \frac{n-k_n}{n}  \binom{n}{k_n}\mu(0)^{k_n}\mu(d+1) (1-\mu(0))^{n-k_n-1}\mathbb{P} \left(W'_{n-k_n-1}=k_n-d-1  \right).
\end{align}
Combining with Equation \eqref{Markov} and Corollary \ref{propTW}, we get
\begin{align*}
\mathbb{P}\left(\sum_{d\geq2}d\phi_{d+1}\left(R_{n}^{k_n} \right) \geq \epsilon k_n\right) & \leq \frac{1}{\epsilon k_n} \sum_{d\geq2}d (n-k_n) \frac{\mu(d+1)}{1-\mu(0)} \frac{\mathbb{P} \left(W'_{n-k_n-1}=k_n-d-1  \right)}{\mathbb{P} \left(W'_{n-k_n}=k_n-1  \right)} \\ 
    & = \frac{1}{\epsilon k_n} \sum_{d=2}^{k_n-1}d (n-k_n) \frac{\mu(d+1)}{1-\mu(0)} \frac{\mathbb{P} \left(W'_{n-k_n-1}=k_n-d-1  \right)}{\mathbb{P} \left(W'_{n-k_n}=k_n-1  \right)},
\end{align*}
where the sum reduces to $d \leq k_n-1$ since $W'_{n-k_n-1} \geq 0$.
We now estimate the ratio of probabilities. We apply
Corollary \ref{corKM}.(ii) to the random walk $W'$ with
generating function $G$, with $r=1$ and $k=-d-1$. Since
$2\le d\le k_n-1$, we have $-k_n\le -d-1\le0$, and therefore,
uniformly for $2\le d\le k_n-1$,
\[\mathbb{P}(W'_{n-k_n-1}=k_n-d-1)
\le
\frac{C}{\sqrt{k_n}}
\frac{G(b_n)^{n-k_n-1}}{b_n^{k_n-d-1}}.\]
On the other hand, Corollary \ref{corKM}(iii), applied
with $r=0$ and $s=1$, gives
\[\mathbb{P}(W'_{n-k_n}=k_n-1)
\underset{n \to \infty}{\sim} \frac{1}{\sqrt{2\pi k_n}}
\frac{G(b_n)^{n-k_n}}{b_n^{k_n-1}}.\]
Consequently, there exists $C'>0$ such that, for all large enough $n$
and uniformly for $2\le d\le k_n-1$,
$$
\frac{\mathbb{P}(W'_{n-k_n-1}=k_n-d-1)}
{\mathbb{P}(W'_{n-k_n}=k_n-1)}
\le
C'\frac{b_n^d}{G(b_n)}.
$$
Plugging this bound into the previous estimate yields
$$
\mathbb{P}\left(
\sum_{d\ge2} d\varphi_{d+1}(R_n^{k_n})\ge \varepsilon k_n
\right)
\le
\frac{C'}{\varepsilon k_n}
\sum_{d=2}^{k_n-1}
d(n-k_n)\frac{\mu(d+1)}{1-\mu(0)}
\frac{b_n^d}{G(b_n)}.
$$
Since $\mu(d+1)/(1-\mu(0))=\xi(d)\le1$, $G(b_n)\to G(0)>0$, and
$b_n\to0$, we get, for $n$ large enough,
$$
\mathbb{P}\left(
\sum_{d\ge2} d\varphi_{d+1}(R_n^{k_n})\ge \varepsilon k_n
\right)
\le
\frac{C''(n-k_n)}{\varepsilon k_n}
\sum_{d\ge2} d b_n^d.
$$
But $\sum_{d\ge2} d b_n^d=b_n^2(2-b_n)/(1-b_n)^2=O(b_n^2)$, and by
Corollary \ref{corKM}(i), $b_n=O(k_n/n)$. Therefore,
$$
\frac{n-k_n}{k_n}\sum_{d\ge2}d b_n^d
=
O\left(\frac{n}{k_n}b_n^2\right)
=
O\left(\frac{k_n}{n}\right)
\underset{n\to\infty}{\longrightarrow}0.
$$
This proves the lemma.
\end{proof}

We now turn to the proof of Proposition \ref{prop_size_R}.
\begin{proof}[Proof of Proposition \ref{prop_size_R}]
On one hand,
\begin{align*}
    \sum_{d\geq 1}(d-1)\phi_{d}\left(R_n^{k_n}\right) & =  \sum_{d\geq 1}d\phi_{d}\left(R_n^{k_n}\right) -  \sum_{d\geq 1}\phi_{d}\left(R_n^{k_n}\right) = |R_n^{k_n}| -1 - \left(|R_n^{k_n}|-k_n\right) = k_n -1.
\end{align*}
On the other hand, by Lemma \ref{prop_deg_Rnk} we have with high probability \[\sum_{d\geq 1}(d-1)\phi_{d}\left(R_n^{k_n}\right) = \phi_{2}\left(R_n^{k_n}\right) + \sum_{d\geq 2}d\phi_{d+1}\left(R_n^{k_n}\right) = \phi_{2}\left(R_n^{k_n}\right) + o(k_n). \]
Thus, we obtain that \[\phi_{2}\left(R_n^{k_n}\right) \underset{n \to \infty}{\sim} k_n \qquad \text{whp}. \]
Finally, the number of edges of $R_n^{k_n}$ is given by \[\sum_{d\geq 0}d\phi_{d}\left(R_n^{k_n}\right) = 2\phi_{2}\left(R_n^{k_n}\right) + \sum_{d\geq 2}(d+1)\phi_{d+1}\left(R_n^{k_n}\right) .\]
Since $d+1 \leq 2d$ for every $d\geq 2$, we apply Lemma \ref{prop_deg_Rnk} to the last term, which implies that, with high probability, \[|R_n^{k_n}|\underset{n \to \infty}{\sim}2 k_n.\]
This completes the proof of Proposition \ref{prop_size_R}.
\end{proof}

\section{Degrees in $T^{k_n}_n$} \label{sec4}
Recall that for any tree $a \in \mathbb{T}_n$, we denote by $\deg_{\max}(a)$ its maximal outdegree, that is,
\[
\deg_{\max}(a) \coloneqq \max_{u \in a} c_u(a).
\]
\begin{table}[htbp]\caption{Table of the main notation introduced in Section \ref{sec4}, and used later.}
\centering
\begin{tabular}{c c p{12cm} }
\toprule
$\deg_{\max}(a)$ && $\max_{u \in a} c_u(a)$, for a tree $a \in \mathbb{T}_i$ \\
$\xi(i), p(i)$  &&$\mu(i+1)/(1-\mu(0)), \mathbb{P}(\xi = i)$, $\forall i \geq 0 $ \\
$\xi_d(i), p_d(i)$ && $\mu(i+1)/(\sum_{i=0}^{d-1}\mu(i+1))$, $\mathbb{P}(\xi_d = i)$, $ \forall i \in [\![0,d-1]\!] $ \\
$G, G_d$ && the generating function of $\xi$ and $\xi_d$  \\
$W^{[d]}_n$ && $W_n$ conditioned to have increments in $[\![0,d-1]\!]$ \\
\bottomrule
\end{tabular}
\end{table}
In this section, we prove Theorem \ref{Th_degrees}. The proof is divided into three parts, corresponding to the three statements. We start with the first one.
\subsection{Probability that degrees in $T_n^{k_n}$ are at most $d$}
Let $d \in \mathbb{Z}_{\geq 2}$ and assume that $k_n=o(n^{\frac{d-1}{d}})$. We aim to show that
\[\mathbb{P}\left(\mathrm{deg_{max}}(T_n^{k_n}) \leq d\right)\xrightarrow[n\to\infty]{}1.\]  By definition of $T_n^{k_n}$, we first rewrite
\[\mathbb{P}\left(\mathrm{deg_{max}}(T_n^{k_n}) \leq d \right) = \frac{\mathbb{P}\left(\mathrm{deg_{max}}(T) \leq d, T \in \mathbb{T}_{n}^{k_n} \right)}{\mathbb{P}\left( T \in \mathbb{T}_{n}^{k_n} \right)}. \]
Using the coding of $T$ by its \L ukasiewicz path (Proposition \ref{propcodingluka}), let $W$ be the random walk started at $W_0=0$ with i.i.d. increments $(X_{i})_{i \in \mathbb{Z}_{\geq 0}}$ distributed according to the law $\nu$ on $\mathbb{Z}_{\geq -1}$ defined by $\nu(\cdot) = \mu(\cdot+1). $ Then, $\mathbb{P}(\mathrm{deg_{max}}(T) \leq d, T \in \mathbb{T}^{k_n}_{n})$ is equal to 
\[\mathbb{P}\left(W_{n}=-1,  W_{i} \geq 0 \,\forall i \in [\![1,n-1]\!], \bigl|\{i : X_{i}=-1\}\bigr|=k_n, \widehat{X}_i \leq d-1 \, \forall i \in [\![1,n-k_n]\!]\right),\]
where we recall that $(\widehat{X}_i)_{i \geq 1}$ denotes the sequence obtained from $(X_i)_{i \geq 1}$ by removing all increments equal to $-1$.
Then, applying Proposition \ref{classiquedecomposition} with $f(x_1,\ldots,x_{n-k_n}) = \1_{x_i \leq d-1 \,\forall i \in [\![1,n-k_n]\!]} $, we obtain 
\[\mathbb{P}\!\left(\mathrm{deg_{max}}(T) \leq d, T \in \mathbb{T}^{k_n}_{n}\right)=\frac{1}{n} \mathbb{P}\left(B_{n}=k_n\right) \mathbb{P}\!\left(W'_{n-k_n} \! = k_n-1, Y_i \leq d-1 \,\, \forall i \! \in \! [\![1,n\!-\!k_n ]\!]\right),\] where $B_{n}$ is the sum of $n$ independent Bernoulli random variables with parameter $\mu(0)$ and $W'$ is a random walk starting at $0$ with i.i.d. increments $(Y_{i})_{i \in \mathbb{Z}_{\geq 0}}$ distributed according to $\xi$. 
 By independence of the $(Y_{i})_{i \in \mathbb{Z}_{\geq 0}}$, we have
\begin{align*}
      & \mathbb{P}\left(W'_{n-k_n}   = k_n-1, Y_i \leq d-1 \,\, \forall i  \in [\![1,n-k_n ]\!]\right) \\  & \qquad  = \mathbb{P}(Y_1 \leq d-1)^{n-k_n} \mathbb{P}\left(W'_{n-k_n}  = k_n-1  \mid Y_i \leq d-1 \,\, \forall i  \in  [\![1,n-k_n ]\!] \right) \\ 
    & \qquad = \mathbb{P}(Y_1 \leq d-1)^{n-k_n}\mathbb{P}\left(W^{[d]}_{n-k_n}  = k_n-1\right),
\end{align*}
where $W^{[d]}$ is equal in distribution to the random walk $W$ conditioned to have increments in $[\![0,d-1]\!]$.
Hence, we get
\begin{equation*}
 \mathbb{P}\left(\mathrm{deg_{max}}(T) \leq d, T \in \mathbb{T}^{k_n}_{n} \right) = \frac{1}{n}\mathbb{P}\bigl(B_{n}=k_n\bigr)\mathbb{P}(Y_1 \leq d-1)^{n-k_n}\mathbb{P}\bigl(W^{[d]}_{n-k_n} = k_n-1\bigr).
\end{equation*}
Finally, using Corollary \ref{propTW} we obtain the following identity
\begin{equation}\label{eqdegmaxR}
    \mathbb{P}\left(\mathrm{deg_{max}}(T_n^{k_n}) \leq d  \right) = \mathbb{P}(Y_1 \leq d-1)^{n-k_n} \frac{\mathbb{P}\bigl(W^{[d]}_{n-k_n} = k_n-1\bigr)}{\mathbb{P}\bigl(W'_{n-k_n} = k_n-1\bigr)}.
\end{equation}
To estimate $\mathbb{P}(\mathrm{deg_{max}}(T_n^{k_n}) \leq d ) $, we need to analyze the asymptotic behavior of probabilities $\mathbb{P}(W'_{n-k_n} = k_n-1)$ and $\mathbb{P}(W^{[d]}_{n-k_n} = k_n-1)$. The key idea is that obtaining their explicit asymptotic expansions seems quite difficult, but we will manage to estimate the asymptotics of their ratios.

\subsubsection{$W$ conditioned to have nonnegative steps}\label{secW'}
We first study the term  $\mathbb{P}(W'_{n-k_n} = k_n-1)$. Recall that $\xi$ is the probability distribution on $\mathbb{Z}_{\geq 0}$ defined by \[\xi \coloneqq \frac{\mu(\cdot+1)}{1-\mu(0)}\] and let $G$ denote its generating function. Since $\mu(1)>0$ and $\mu(2)>0$, we have $\xi(0)>0$ and $\xi(1)>0$. Thus, applying Corollary \ref{corKM}(iii) to $W'$, we get that
\begin{equation}\label{eqW'}
    \mathbb{P}\left(W'_{n-k_n} = k_n-1\right) \underset{n \to \infty}{\sim} \frac{1}{\sqrt{2\pi k_n}}\frac{G\left(b_{n}\right)^{n-k_n}}{b_{n}^{k_n-1}},
\end{equation}
where
\begin{equation} \label{eqb_n1}
b_{n}= \frac{k_n}{n-k_n}\,\frac{G(b_{n})}{G'(b_{n})}
\underset{n \to \infty}{\sim} \frac{k_n}{n}\,\frac{\mathbb{P}(\xi  = 0)}{\mathbb{P}(\xi  = 1)}.
\end{equation}
To simplify the notation, we set $q_n \coloneqq k_n/(n-k_n)$ and \(p(i) \coloneqq \mathbb{P}(\xi = i)\) for every integer \(i\).
In what follows, we require a more refined expansion of $b_n$. We first derive its second order expansion. Introduce the notation $\widetilde{p}(i)\coloneqq (i+1)p(i+1)$, so that $G'(b_n)= \sum_{i\geq 0}\widetilde{p}(i)b_n^i$. Expanding $G$ and $G'$ at $0$, we obtain
\[b_n=q_n\left( \sum_{i\geq 0} p(i) b_n^i\right) \left( \sum_{i\geq 0} \widetilde{p}(i) b_n^i\right)^{-1} \!= q_n \, \frac{ p(0) + p(1) b_n }{\widetilde{p}(0) + \widetilde{p}(1) b_n }  (1 + o(b_n)).\]
Substituting the equivalent \eqref{eqb_n1} into this expression yields
\begin{align*}
   b_n &= q_n \, \frac{ p(0) + p(0)q_n + o\left(q_n\right) }{p(1) + \frac{2p(0)p(2)}{p(1)} q_n+ o\left(q_n\right) } \left(1 + o\left(q_n\right)\right) \\
   & = q_n \left( p(0) + p(0)q_n + o\left(q_n\right) \right) \left( \frac{1}{p(1)} - \frac{2p(0)p(2)}{p(1)^3}q_n + o\left(q_n\right)\right) \left(1 + o\left(q_n\right)\right).
\end{align*}
We conclude that $b_n$ admits the second-order expansion \begin{equation}\label{eqb_n}
    b_n = c_1q_n + c_2q^2_n + o\left(q^2_n\right), \quad \text{with } c_1 \coloneqq \frac{p(0)}{p(1)} \quad \text{and }  c_2 \coloneqq \frac{p(0)}{p(1)}\left(1 - \frac{2p(0)p(2)}{p(1)^2} \right).
\end{equation}
By iterating this bootstrapping argument, one can in fact obtain an expansion of $b_n$ to order $d$: there exist constants $c_1,\ldots,c_d$ such that
\begin{equation}\label{eqb_nordred}
    b_{n} = c_1q_n + \cdots + c_d q^d_n + o\left(q^d_n\right).
\end{equation}
We now derive asymptotic expansions for the terms appearing in \eqref{eqW'}. On the one hand, a Taylor expansion yields
\begin{align*}
    \log G(b_n)^{n-k_n}
&= (n-k_n) \left( \log{\left(p(0)\right)} + \log{\left( 1 +\frac{p(1)}{p(0)}b_n+ \cdots + \frac{p(d)}{p(0)}b_n^d  + o\left( b_n^d\right)\right)} \right) \\
    &  = (n-k_n) \left(\log{\left(p(0)\right)} +  \alpha_1q_n + \cdots + \alpha_dq^d_n + o\left( q^d_n\right)\right).
   \end{align*}
for suitable constants $\alpha_1,\ldots,\alpha_d$. Hence,
\[G(b_n)^{n-k_n}  = p(0)^{n-k_n} \exp{\left(\alpha_1(n-k_n)q_n + \cdots + \alpha_d(n-k_n)q^d_n + o\left((n-k_n)q^d_n\right)\right)},\]
On the other hand, using again the expansion of $b_n$ \eqref{eqb_nordred}, we obtain 
\begin{align*}\label{eqbnpuisd}
    b_n^{k_n-1}\! & = \!\left(c_1q_n\right)^{k_n-1} \exp{ \left(\beta_1(k_n-1)q_n+ \cdots + \beta_{d-1}(k_n-1)q^{d-1}_n + o \left((k_n-1)q^{d-1}_n \!\right) \right)},
\end{align*}
where $\beta_1,\ldots,\beta_{d-1} \in \mathbb{R}$.

\subsubsection{$W$ conditioned to have increments in $[\![0,d-1]\!]$}\label{secWd}
 We now study the probability $\mathbb{P}(W^{[d]}_{n-k_n} = k_n-1)$. Define the probability distribution $\xi_d$ on $[\![0,d-1]\!]$ by \[  \xi_d(i) \coloneqq  \frac{\mu(i+1)}{1-\mu(0)-\sum_{i=d}^{\infty}\mu(i+1)} = \frac{\mu(i+1)}{\sum_{i=0}^{d-1}\mu(i+1)}, \quad \text{for all } i \in [\![0,d-1]\!] \] and let $G_d$ denote its generating function.
 Note that $\mu(1)>0$ and $\mu(2)>0$ imply $\xi_d(0)>0$ and $\xi_d(1)>0$. Applying Corollary \ref{corKM}(iii) to  $W^{[d]}$, we obtain
\begin{equation}\label{eqWd}
    \mathbb{P}\left(W^{[d]}_{n-k_n} = k_n-1\right) \underset{n \to \infty}{\sim} \frac{1}{\sqrt{2\pi k_n}}\frac{G_d\left(b_{d,n}\right)^{n-k_n}}{b_{d,n}^{k_n-1}}, 
\end{equation}
where $b_{d,n}$ is defined by  \[b_{d,n}= q_n\,\frac{G_d(b_{d,n})}{G'_d(b_{d,n})}
\underset{n \to \infty}{\sim} \frac{k_n}{n}\,\frac{\mathbb{P}(\xi_d  = 0)}{\mathbb{P}(\xi_d  = 1)}.\]
For convenience, we set \(p_d(i) \coloneqq \mathbb{P}(\xi_d = i)\) for all integers $i$.
Then \[G_d(x) = \sum_{i=0}^{d-1}p_d(i)x^i = c_{\mu}\sum_{i=0}^{d-1}p(i)x^i \quad \text{and} \quad G'_d(x) = \sum_{i=1}^{d-1}ip_d(i)x^{i-1} = c_{\mu}\sum_{i=0}^{d-2}\widetilde{p}(i)x^i,\]
where \[c_{\mu} \coloneqq \frac{1-\mu(0)}{1-\mu(0)-\sum{i\geq d} \mu(i+1)}.\] 
In particular, $b_{d,n}$ satisfies
\begin{equation} \label{formule_b}
    b_{d,n}=q_n\left( \sum_{i= 0}^{d-1} p(i) b_{d,n}^i\right) \left( \sum_{i= 0}^{d-2} \widetilde{p}(i) b_{d,n}^i\right)^{-1}.
\end{equation}
By iterating the bootstrapping argument of the previous subsection, we deduce that $b_{d,n}$ admits the same expansion as $b_n$ up to order $d-1$, namely
\begin{equation} \label{eqb_ordre_d-1}
    b_{d,n} = c_1q_n + \cdots + c_{d-1} q^{d-1}_n + o\left(q^{d-1}_n \right),
\end{equation}
where $c_1,\ldots,c_{d-1}$ are the same constants as in the expansion of $b_n$.
We now determine the expansion of $b_{d,n}$ at order $d$ by comparison with that of $b_n$. At this order, we have 
\[ b_{n} = q_n \left(\sum_{i=0}^{d-1} p(i) b_{n}^i + o\left(b_n^{d-1}\right)\right) \left( \sum_{i=0}^{d-1} \widetilde{p}(i)b_n^i + o\left(b_n^{d-1}\right) \right)^{-1}\]
whereas \begin{equation}\label{bdn}
    b_{d,n} = q_n  \left(\sum_{i=0}^{d-1} p(i) b_{d,n}^i\right) \left( \sum_{i=0}^{d-2} \widetilde{p}(i)b_{d,n}^i  \right)^{-1}.
\end{equation}
Observe that the denominator in the expression of $b_n$ contains the additional term $dp(d)b_n^{d-1}$ compared to that of $b_{d,n}$. Consequently, injecting Equation \eqref{eqb_ordre_d-1} into the formula of $b_{d,n}$ \eqref{bdn}, we finally obtain
\begin{equation}\label{eqbdnordred}
 b_{d,n} = c_1q_n + \cdots + c_{d-1} q^{d-1}_n  + \left( \! c_{d} + \frac{dp(0)^{d}p(d)}{p(1)^{d+1}} \! \right) q^{d}_n   + o\left(q^{d}_n \right).
\end{equation}
Note that\[\frac{p(i)}{p(0)}=\frac{p_d(i)}{p_d(0)} \, \forall i \in [\![1,d-1]\!] \qquad \text{while} \qquad \frac{p_d(d)}{p_d(0)}=0 \text{ and } \frac{p(d)}{p(0)}=\frac{\mu(d+1)}{\mu(1)}.\]
Therefore, there exist constants $\alpha^{[d]}_1,\ldots,\alpha^{[d]}_d$ and $\beta^{[d]}_1,\ldots,\beta^{[d]}_{d-1}$ such that 
\begin{align*}
    \log{G_d(b_{d,n})^{n-k_n}} & \!= (n-k_n) \left( \log{\left(p_d(0)\right)} + \log{\left( 1 +\frac{p(1)}{p(0)}b_{d,n}+ \cdots + \frac{p(d-1)}{p(0)}b_{d,n}^{d-1} \right)} \! \right) \\
    &= (n-k_n) \! \left(\log{\left(p_d(0)\right)} +  \alpha^{[d]}_1 q_n + \cdots + \alpha^{[d]}_d q^{d}_n  + o\left( q^{d}_n \right) \!\right),
\end{align*}
so that  \begin{align*}
    G_d(b_{d,n})^{n-k_n} & = p_d(0)^{n-k_n}\exp{\left(\alpha^{[d]}_1(n-k_n)q_n + \cdots + \! \alpha^{[d]}_d (n-k_n)q^d_n  + o\left( (n-k_n)q^d_n\right)\right)},
\end{align*}
and 
\begin{align*}
    b_{d,n}^{k_n-1} & =  \left(c_1q_n\right)^{k_n-1} \exp{ \left( \beta^{[d]}_1(k_n-1)q_n+ \cdots + \beta^{[d]}_{d-1}(k_n-1)q^{d-1}_n +  o \left( k_nq^{d-1}_n\right)  \right)}.
\end{align*}
Combining with Equations \eqref{eqb_nordred} and \eqref{eqbdnordred}, we deduce that  
 \begin{align} 
\forall i \in [\![1,d-1]\!] , \, \alpha_i  & = \alpha^{[d]}_i \qquad \text{and} \qquad  \alpha^{[d]}_d = \alpha_d - \frac{p(d)c_1^d}{p(0)} +\frac{dp(0)^{d-1}p(d)}{p(1)^d}, \label{defalpha} \\
\forall i \in [\![1,d-2]\!] , \,\beta_i & = \beta^{[d]}_i \qquad \text {and} \qquad  \beta^{[d]}_{d-1} = \beta_{d-1} + \frac{dp(0)^{d}p(d)}{c_1p(1)^{d+1}}. \label{defbeta}
\end{align}

\subsubsection{Conclusion}
We now return to the estimation of $\mathbb{P}(\mathrm{deg_{max}}(T_n^{k_n}) \leq d )$. Recall that, by Equation \eqref{eqdegmaxR},
\begin{equation*}
    \mathbb{P}\left(\mathrm{deg_{max}}(T_n^{k_n}) \leq d  \right) = 
    \mathbb{P}(Y_1 \leq d-1)^{n-k_n} \frac{\mathbb{P}\bigl(W^{[d]}_{n-k_n} = k_n-1\bigr)}{\mathbb{P}\bigl(W'_{n-k_n} = k_n-1\bigr)}.
\end{equation*}
Consequently, using Equations \eqref{eqW'} and \eqref{eqWd}, we obtain the asymptotic equivalence 
\begin{equation}\label{eqWd-W-Gd-G}
\frac{\mathbb{P}\left(W^{[d]}_{n-k_n} = k_n-1\right)}{\mathbb{P}\left(W'_{n-k_n} = k_n-1\right)}\underset{n \to \infty}{\sim}  \frac{G_d\left(b_{d,n}\right)^{n-k_n}}{G\left(b_{n}\right)^{n-k_n}}\frac{b_{n}^{k_n-1}}{b_{d,n}^{k_n-1}}.
\end{equation}
Combining with the definition of $q_n$ and expansions obtained in the previous subsections, we have
\begin{align*}
   \frac{G_d\left(b_{d,n}\right)^{n-k_n}}{G\left(b_{n}\right)^{n-k_n}}\frac{b_{n}^{k_n-1}}{b_{d,n}^{k_n-1}} = &
   \left(\frac{p_d(0)}{p(0)}\right)^{n-k_n}   \exp{   \left(  \left(\alpha^{[d]}_d-\alpha_d \right)\frac{k_n^d}{(n-k_n)^{d-1}} + o \left( \frac{k_n^d}{(n-k_n)^{d-1}} \right) \right)} \\
   & \times \exp{ \left( -\left( \beta^{[d]}_{d-1}  - \beta_{d-1}\right) \frac{k_n^{d-1}(k_n-1)}{(n-k_n)^{d-1}} + o\left( \frac{k_n^d}{(n-k_n)^{d-1}}\right) \right)}.
\end{align*}
In particular, if $k_n=o\left( n^{\frac{d-1}{d}}\right)$, we get
\[
\frac{G_d\left(b_{d,n}\right)^{n-k_n}}{G\left(b_{n}\right)^{n-k_n}}\frac{b_{n}^{k_n-1}}{b_{d,n}^{k_n-1}}= \left(\frac{p_d(0)}{p(0)}\right)^{n-k_n} \! \! \! \! \! \! \! \! \exp{\left(-\frac{p(0)^{d-1}p(d)}{p(1)^d} \frac{k_n^{d}}{n^{d-1}} + o\left( \frac{k_n^d}{n^{d-1}}\right) \right)} \underset{n \to \infty}{\sim}  \left(\frac{p_d(0)}{p(0)}\right)^{n-k_n}.\]
Combining with the asymptotic equivalent \eqref{eqWd-W-Gd-G}, we obtain
\[\frac{\mathbb{P}\left(W^{[d]}_{n-k_n} = k_n-1\right)}{\mathbb{P}\left(W'_{n-k_n} = k_n-1\right)} \underset{n \to \infty}{\sim} \left(\frac{p_d(0)}{p(0)}\right)^{n-k_n}.\]
Observing that 
\[\left(\frac{p_d(0)}{p(0)}\right)^{n-k_n} = \left(\frac{1-\mu(0)}{\sum_{i = 0}^{d-1} \mu(i+1)}\right)^{n-k_n} = \frac{1}{\mathbb{P}(Y_1 \leq d-1)^{n-k_n}},\]
we finally get
\begin{equation*} \mathbb{P}\left(\mathrm{deg_{max}}(T_n^{k_n}) \leq d  \right) \xrightarrow[n\to\infty]{}1,
\end{equation*}
which completes the proof of the first item of Theorem \ref{Th_degrees}.

\subsection{Probability that $T_n^{k_n}$ has $m$ vertices with outdegree $d+1$}
We now prove the second statement of Theorem \ref{Th_degrees}. Fix $d \geq 2$, $m \geq 0$, $c>0$ and assume $k_n\underset{n \to \infty}{\sim} cn^{\frac{d-1}{d}}$. In particular, $k_n = o(n^{\frac{d}{d+1}})$, so the first item of Theorem \ref{Th_degrees} applied with $d+1$ yields
\begin{equation*}
   \mathbb{P}\left(\mathrm{deg_{max}}(T_n^{k_n}) \leq d+1 \right)\xrightarrow[n\to\infty]{}1.
\end{equation*}
Thus, it is enough to show that \[\mathbb{P}\! \left(\phi_{d+1}\left(T_n^{k_n}\right)=m, \mathrm{deg_{max}}(T_n^{k_n}) \leq d+1 \right)\! \underset{n \to \infty}{\sim} \frac{1}{m!} \! \left(c^d\frac{\mu(d+1)\mu(1)^{d-1}}{\mu(2)^d} \right)^m \! \! \!\exp{\! \left(\!-c^d\frac{\mu(1)^{d-1}\mu(d+1)}{\mu(2)^d} \right)}\]
to complete the proof of Theorem \ref{Th_degrees}.2.
First, we write \[\mathbb{P}\left(\phi_{d+1}\left(T_n^{k_n}\right)=m, \mathrm{deg_{max}}(T_n^{k_n}) \leq d+1 \right)  = \frac{\mathbb{P}\left(\phi_{d+1}\left(T\right)=m, \mathrm{deg_{max}}(T) \leq d+1, T \in \mathbb{T}_{n}^{k_n} \right)}{\mathbb{P}\left( T \in \mathbb{T}_{n}^{k_n} \right)}. \]
Using the coding of $T$ by its \L ukasiewicz path (Proposition \ref{propcodingluka}), the numerator can be expressed as
\[\mathbb{P}\!\left(W_{n}\!=\!-1,  W_{i} \geq 0 \,\forall i \!\in \![\![1,n\!-\!1]\!], \bigl|\{i:  X_{i}=\!-1\}\bigr|\!= \!k_n,\bigl|\{i:  \widehat{X}_{i}=\!d\}\bigr|\! = \!m,\widehat{X}_i \leq d \, \forall i \! \in \! [\![1,n\!-\!k_n]\!]\right)\!.\]
Applying Proposition \ref{classiquedecomposition} with \[f(x_1,\ldots,x_{n-k_n}) = \1_{\left|\{i:  x_{i}=\!d\}\right|=m} \1_{ x_i \leq d \, \forall i \in [\![1,n-k_n]\!]},\] we obtain that the latter probability is equal to
\begin{align*}
   \frac{1}{n}\binom{n}{k_n}\mu(0)^{k_n}(1-\mu(0))^{n-k_n} \mathbb{P}\Bigl(W'_{n-k_n} = k_n-1,\ |\{i : Y_i = d\}|=m,\ Y_i \leq d \ \forall i \in \! [\![1,n\!-\!k_n ]\!]\Bigr).
\end{align*}
Conditioning on the number of indices such that $Y_i=d$, we obtain that  $\mathbb{P}(W'_{n-k_n} = k_n-1,\ |\{i : Y_i = d\}|=m,\ Y_i \leq d \, \forall i \in \! [\![1,n\!-\!k_n ]\!])$ equals
\begin{align*}
  \binom{n-k_n}{m} \mathbb{P}(Y_1=d)^m \mathbb{P}(Y_1 \leq d-1)^{n-k_n-m} \mathbb{P}\Bigl(W^{[d]}_{n-k_n-m} = k_n-1-md\Bigr).
\end{align*}
Combining the above identities and using Corollary \ref{propTW}, we obtain the following identity
\begin{align}\label{eqdegRm}
     & \mathbb{P}\left(\phi_{d+1}\left(T_n^{k_n}\right)=m, \mathrm{deg_{max}}(T_n^{k_n}) \leq d+1 \right) \notag \\ & \qquad =  \binom{n-k_n}{m} \mathbb{P}(Y_1=d)^m \mathbb{P}(Y_1 \leq d-1)^{n-k_n-m} \frac{\mathbb{P}\Bigl(W^{[d]}_{n-k_n-m} = k_n-1-md\Bigr)}{\mathbb{P}\bigl(W'_{n-k_n}\!\! = k_n-1\bigr)}
\end{align}
It remains to analyze the asymptotic behavior of the ratio
\[
\frac{\mathbb{P}\bigl(W^{[d]}_{n-k_n-m} = k_n-1-md\bigr)}
{\mathbb{P}\bigl(W'_{n-k_n} = k_n-1\bigr)}.
\]
The asymptotic behavior of the denominator was established in Section~\ref{secW'}. By Corollary \ref{corKM}.(iii), the numerator is asymptotically equivalent to \[  \frac{1}{\sqrt{2\pi k_n}}\frac{G_d(b_{d,n})^{n-k_n-m}}{b_{d,n}^{k_n-1-md}},\] and can therefore be handled in exactly the same way as in the case $m=0$, treated in Section~\ref{secWd}. This yields
\begin{align*}
    G_d(b_{d,n})^{n-k_n-m} & \!= p_d(0)^{n-k_n-m}\exp{\left(\!\alpha^{[d]}_1 (n\!-\!k_n\!-\!m)q_n+ \cdots + \! \alpha^{[d]}_d(n\!-\!k_n\!-\!m)q^d_n + o\left( (n\!-\!k_n\!-\!m)q^d_n \right)\!\right)},
\end{align*}
and 
\begin{align*}\label{eqbnmpuisd}
    b_{d,n}^{k_n-1-md} \!& = \! \left(c_1q_n\right)^{k_n-1-md}   \!\exp{ \!\left( \beta^{[d]}_1(k_n\!-\!1\!-\!md)q_n + \! \cdots \!+ \beta^{[d]}_{d-1}(k_n\!-\!1\!-\!md)q^{d-1}_n \! +  o \!\left( \!(k_n\!-\!1\!-\!md)q^{d-1}_n \right) \! \right)},
\end{align*}
where we recall that $q_n = k_n/(n-k_n)$ and $\alpha^{[d]}_1,\ldots,\alpha^{[d]}_d,\beta^{[d]}_1,\ldots,\beta^{[d]}_{d-1}$ are defined in Equations \eqref{defalpha} and \eqref{defbeta}.
Hence, combining with Equations \eqref{eqW'} and \eqref{eqWd}, we deduce that if $k_n\underset{n \to \infty}{\sim} cn^{\frac{d-1}{d}}$
\[\frac{\mathbb{P}\bigl(W^{[d]}_{n-k_n-m} = k_n-1-md\bigr)}{\mathbb{P}\bigl(W'_{n-k_n} = k_n-1\bigr)} \underset{n \to \infty}{\sim} \frac{\left(p_d(0)\right)^{n-k_n-m}}{\left(p(0)\right)^{n-k_n}} \left(c_1 q_n\right)^{md}  \exp{\left(-c^d\frac{p(0)^{d-1}p(d)}{p(1)^d} \right)}. \]
Noticing that
\[ \binom{n-k_n}{m} \mathbb{P}(Y_1=d)^m \mathbb{P}(Y_1 \leq d-1)^{n-k_n-m} = \binom{n-k_n}{m}\frac{\mu(d+1)^m}{\left(1-\mu(0) \right)^{n-k_n}}\left(\sum_{i=0}^{d-1}\mu(i+1)\right)^{n-k_n-m} \]
and 
\[\frac{\left(p_d(0)\right)^{n-k_n-m}}{\left(p(0)\right)^{n-k_n}} \left(c_1q_n\right)^{md}= \left(\frac{\mu(1)}{\sum_{i=0}^{d-1}\mu(i+1)}\right)^{n-k_n-m}\left(\frac{1-\mu(0)}{\mu(1)}\right)^{n-k_n} \left(\frac{\mu(1)}{\mu(2)}\frac{k_n}{n-k_n}\right)^{md},\]
We deduce from Equation \eqref{eqdegRm} that $\mathbb{P}(\phi_{d+1}\left(T_n^{k_n}\right)=m, \mathrm{deg_{max}}(T_n^{k_n}) \leq d\!+\!1 )$ is asymptotically equivalent to 
\begin{align*}
& \frac{1}{m!} \frac{(n-k_n)!}{(n-k_n-m)!}\frac{k_n^{md}}{(n-k_n)^{md}}\left(\frac{\mu(d+1)\mu(1)^{d-1}}{\mu(2)^d} \right)^m \exp{\left(-c^d\frac{\mu(1)^{d-1}\mu(d+1)}{\mu(2)^d} \right)} \\
& \qquad \underset{n \to \infty}{\sim} \frac{1}{m!} \left(c^d\frac{\mu(d+1)\mu(1)^{d-1}}{\mu(2)^d} \right)^m \exp{\left(-c^d\frac{\mu(1)^{d-1}\mu(d+1)}{\mu(2)^d} \right)},
\end{align*}
which concludes the proof of Theorem \ref{Th_degrees}.2.

\subsection{Convergence in $L^p$}
We now prove the third item of Theorem \ref{Th_degrees}. Fix $d \in \mathbb{Z}_{\geq 2}$ and assume $k_n^d/n^{d-1} \rightarrow + \infty$ as $n\rightarrow +\infty$. The argument relies on two lemmas. First, Lemma \ref{cvg_en_proba} establishes a weak law of large numbers for the number of vertices in $T^{k_n}_n$ with outdegree $d+1$. 
Second, Lemma \ref{borneLp} ensures that this sequence is bounded in $L^p$ for every $p \geq 1$. Hence, the convergence holds in $L^p$ for every $p \geq 1$, and Theorem \ref{Th_degrees}.3 follows.

\begin{lemma}\label{cvg_en_proba}
If $k_n^d/n^{d-1} \rightarrow + \infty$ as $n\rightarrow +\infty$, then we have the following convergence in probability
\[\frac{n^{d-1}}{k_n^d} \phi_{d+1}\left(T_{n}^{k_n}\right)\xrightarrow[n\to\infty]{\mathbb{P}} \frac{\mu(1)^{d-1}\mu(d+1)}{\mu(2)^d}.\]
\end{lemma}
\begin{proof}
Let $d \in \mathbb{Z}_{\geq 2}$, let $\epsilon >0$. If $\mu(d+1)=0$, the claim is trivial. We assume $\mu(d+1) >0$. By the Bienaymé-Tchebychev inequality, we have 
\begin{equation}\label{BT}
    \mathbb{P} \left( \left| \frac{\phi_{d+1}\left(T_{n}^{k_n}\right)}{\mathbb{E}\left[\phi_{d+1}\left(T_{n}^{k_n} \right)\right]}-1 \right| > \epsilon \right) \leq \frac{\mathrm{Var}\left[\phi_{d+1}\left(T_{n}^{k_n} \right)\right]}{\epsilon^2\mathbb{E}\left[\phi_{d+1}\left(T_{n}^{k_n} \right)\right]^2}.
\end{equation}
Thus, it suffices to prove that \[\mathrm{Var}\left[\phi_{d+1}(T_{n}^{k_n})\right] = o\left(\mathbb{E}\left[\phi_{d+1}(T_{n}^{k_n})\right]^2 \,\right).\]
We first recall that by Equation \eqref{espcond},
\begin{align*}
  \mathbb{E}\left[\phi_{d+1}\left(T_{n}^{k_n} \right)\right]
  = \frac{1}{\mathbb{P}\left( T \in \mathbb{T}_{n}^{k_n} \right)} \frac{n-k_n}{n}  \binom{n}{k_n}\mu(0)^{k_n}\mu(d+1) (1-\mu(0))^{n-k_n-1}\mathbb{P} \left(W'_{n-k_n-1}=k_n-d-1  \right).
\end{align*}
Applying Corollary \ref{propTW} and then Corollary \ref{corKM}.(iii), we obtain
\begin{align*}
\mathbb{E}\left[\phi_{d+1}\left(T_{n}^{k_n} \right)\right] & = (n-k_n)\frac{\mu(d+1)}{1-\mu(0)}\frac{\mathbb{P} \left(W'_{n-k_n-1}=k_n-d-1  \right)}{\mathbb{P} \left(W'_{n-k_n}=k_n-1  \right)} \underset{n \to \infty}{\sim} n\frac{\mu(d+1)}{1-\mu(0)}\frac{b^d_n}{G(b_n)} .
\end{align*}
Using Equation \eqref{eqb_n1}, this yields
\begin{equation} \label{esp}
\mathbb{E}\left[\phi_{d+1}\left(T_{n}^{k_n} \right)\right] \underset{n \to \infty}{\sim} \frac{\mu(1)^{d-1}\mu(d+1)}{\mu(2)^d}\frac{k_n^d}{n^{d-1}}.
\end{equation}
We now turn to the second moment. We write 
\begin{align}\label{formespcond2}
    \mathbb{E}\left[\phi_{d+1}(T_{n}^{k_n})^{2} \right] = \frac{ \mathbb{E}\left[\phi_{d+1}(T)^2 \1_{T \in \mathbb{T}_{n}^{k_n}}  \right] }{\mathbb{P}\left( T \in \mathbb{T}_{n}^{k_n} \right)}.
\end{align}
Expanding the square and using the \L ukasiewicz path encoding (Proposition \ref{propcodingluka}), we deduce that  $\mathbb{E}[\phi_{d+1}(T)^2 \1_{T \in \mathbb{T}_{n}^{k_n}} ]$ is equal to
\begin{align*}
    & \sum_{j=1}^{n-1}  \mathbb{E} \left[\1_{\substack{\text{the } j^{th} \text{ vertex of } \\ T \text{ has } d+1 \text{ children} } }\1_{T \in \mathbb{T}_{n}^{k_n}} \right] + \sum_{j_1\neq j_2}  \mathbb{E} \left[\1_{\substack{\text{the } j_1^{th} \text{ and } j_2^{th} \text{ vertex} \\ \text{of } T \text{ have } d+1 \text{ children} } }\1_{T \in \mathbb{T}_{n}^{k_n}} \right] \\
    & = \! \mathbb{E} \!\left[\phi_{d+1}(T) \1_{T \in \mathbb{T}_{n}^{k_n}}  \right] \! + \! \! \sum_{j_1\neq j_2} \mathbb{P}\! \left(W_{n}=-1,  W_i \geq 0 \, \forall i \! \in \! [\![0,n\!-\!1]\!],\bigl|\{i: X_{i}=-1\}\bigr|=k_n,  \widehat{X}_{j_1} \!=\! \widehat{X}_{j_2}\! = \!d  \right) \!.
\end{align*}
Applying Proposition \ref{classiquedecomposition} with \[f(x_1,\ldots,x_{n-k_n}) = \sum_{ j_1  \neq  j_2 \in [\![1,n-k_n]\!]}\1_{x_{j_1} =  x_{j_2}  = d},\] we rewrite the second term as 
\begin{align*}
    \frac{1}{n} \binom{n}{k_n}\mu(0)^{k_n} (1-\mu(0))^{n-k_n}\mathbb{E}\left[ f\left( (Y_i)_{1\leq i \leq n-k_n}\right) \1_{W'_{n-k_n}=k_n-1} \right].
\end{align*}
Moreover, we have by independence \[\mathbb{E}\left[ f\left( (Y_i)_{1\leq i \leq n-k_n}\right) \1_{W'_{n-k_n}=k_n-1} \right] = (n-k_n)(n-k_n-1)\,\mathbb{P}\left(Y_1=Y_2 = d, W'_{n-k_n}=k_n-1 \right),\]
and by the definition of $W'$ and $(Y_i)_{1\leq i \leq n-k_n}$, we obtain 
\begin{align*}
  \mathbb{P}\left(Y_1=Y_2 = d, W'_{n-k_n}\!=k_n\!-\!1 \right) & = \mathbb{P}\left( Y_1=Y_2= d \right)\mathbb{P}\left(W'_{n-k_n}=k_n-1 \mid  Y_1=Y_2 = d \right) \\ 
 & = \left(\frac{\mu(d+1)}{1-\mu(0)}\right)^2\mathbb{P}\left(W'_{n-k_n-2}=k_n-2d-1 \right).
\end{align*} 
Hence, we deduce that $\mathbb{E}[\phi_{d+1}(T)^2 \1_{T \in \mathbb{T}_{n}^{k_n}} ]$ equals
\[\mathbb{E} \!\left[\phi_{d+1}(T) \1_{T \in \mathbb{T}_{n}^{k_n}} \! \right]  +  \frac{(n\!-\!k_n)(n\!-\!k_n\!-\!1)}{n} \binom{n}{k_n}\mu(0)^{k_n}\mu(d+1)^2 (1-\mu(0))^{n-k_n-2} \, \mathbb{P} \! \left( W'_{n-k_n-2}\!=k_n\!-\!2d\!-\!1 \right).\]
Plugging this into Equation \eqref{formespcond2}, we obtain
\begin{equation}\label{espcarré}
\! \mathbb{E} \! \left[\phi_{d+1}\left(T_{n}^{k_n} \right)^2\right] \! = \mathbb{E} \! \left[\phi_{d+1}(R_n^{k_n}) \right] + (n-k_n)(n-k_n-1)\frac{\mu(d\!+\!1)^2}{(1\!-\!\mu(0))^{2}} \frac{\mathbb{P} \!\left(W'_{n-k_n-2}\!=k_n\!-\!2d\!-\!1  \right)}{\mathbb{P} \left(W'_{n-k_n}\!=k_n\!-\!1 \right),}
\end{equation}
where, by Corollary \ref{corKM}.(iii), the second term is asymptotically equivalent to 
\begin{align*}
   n^2\frac{\mu(d+1)^2}{(1-\mu(0))^{2}}\frac{b^{2d}_n}{G(b_n)^2} \underset{n \to \infty}{\sim} \frac{\mu(1)^{2d-2}\mu(d+1)^2}{\mu(2)^{2d}}\frac{k_n^{2d}}{n^{2d-2}}.
\end{align*} 
Therefore,
\begin{align*}
\frac{\mathrm{Var}\left[\phi_{d+1}\left(T_{n}^{k_n} \right)\right]}{\mathbb{E}\left[\phi_{d+1}\left(T_{n}^{k_n} \right)\right]^2} = \frac{\mathbb{E}\left[\phi_{d+1}\left(T_{n}^{k_n} \right)^2\right]}{\mathbb{E}\left[\phi_{d+1}\left(T_{n}^{k_n} \right)\right]^2} - 1,
\end{align*}
where
\begin{align*}
\frac{\mathbb{E}\left[\phi_{d+1}\left(T_{n}^{k_n} \right)^2\right]}{\mathbb{E}\left[\phi_{d+1}\left(T_{n}^{k_n} \right)\right]^2} & = \frac{1}{\mathbb{E}\left[\phi_{d+1}(T_n^{k_n}) \right]} + \frac{n-k_n-1}{n-k_n}\frac{\mathbb{P} \left(W'_{n-k_n-2}=k_n-2d-1 \right)\mathbb{P} \left(W'_{n-k_n}=k_n-1 \right)}{\left(\mathbb{P} \left(W'_{n-k_n-1}=k_n-d-1 \right)\right)^2},
\end{align*}
and \[\frac{n-k_n-1}{n-k_n}\frac{\mathbb{P} \left(W'_{n-k_n-2}=k_n-2d-1 \right)\mathbb{P} \left(W'_{n-k_n}=k_n-1 \right)}{\left(\mathbb{P} \left(W'_{n-k_n-1}=k_n-d-1 \right)\right)^2} \underset{n \to \infty}{\sim}  1, \qquad \mathbb{E}\left[\phi_{d+1}(T_n^{k_n}) \right] \xrightarrow[n\to\infty]{} + \infty.\]
This implies
\[\frac{\mathrm{Var}\left[\phi_{d+1}\left(T_{n}^{k_n} \right)\right]}{\mathbb{E}\left[\phi_{d+1}\left(T_{n}^{k_n} \right)\right]^2} \xrightarrow[n\to\infty]{} 0.\]
and combining this with Equation \eqref{BT} completes the proof.
\end{proof}
\begin{lemma}\label{borneLp} If $k_n^d/n^{d-1} \rightarrow + \infty$ as $n\rightarrow +\infty$, then for every $p \in \mathbb{Z}_{\geq 1}$
    \[\sup_{n\geq 0} \mathbb{E}\left[\left(\frac{n^{d-1}}{k_n^d}\phi_{d+1}\left(T_{n}^{k_n} \right) \right)^p \,\right] < +\infty\]
\end{lemma}
\begin{proof}
Let  $p \in \mathbb{Z}_{\geq 1}$. 
First, we write 
\begin{align}\label{formespcond}
    \mathbb{E}\left[\phi_{d+1}(T_{n}^{k_n})^{p} \right] = \frac{ \mathbb{E}\left[\phi_{d+1}(T)^p \1_{T \in \mathbb{T}_{n}^{k_n}}  \right] }{\mathbb{P}\left( T \in \mathbb{T}_{n}^{k_n} \right)} = \frac{1}{\mathbb{P}\left( T \in \mathbb{T}_{n}^{k_n} \right)} \mathbb{E}\left[ \left( \sum_{j=1}^{n}  \1_{\substack{\text{the } j^{th} \text{ vertex of } T \text{ has } \\ d+1 \text{ children, } T \in \mathbb{T}_{n}^{k_n} } } \right)^p \,\right] \qquad
\end{align}
To simplify the notation, we denote in the whole proof \[ \mathbf{1}_{j} \coloneqq \1_{\substack{\text{the } j^{th} \text{ vertex of } T \text{ has } \\ d+1 \text{ children, } T \in \mathbb{T}_{n}^{k_n} }}.\]
Moreover, if $S(p,r)$ denotes the number of ways to partition a set of size $p$ into $r$ nonempty and  unordered subsets, we have the following formula
\begin{equation}\label{forumuleavecS}
    \mathbb{E}\left[ \left( \sum_{j=1}^{n} \mathbf{1}_{j} \right)^p \,\right] = \sum_{r=1}^{p} S(p,r) \, \mathbb{E}\left[ \sum_{j_1 \neq \cdots \neq  j_r} \mathbf{1}_{j_1} \cdots \mathbf{1}_{j_r} \right].
\end{equation}
For every $r \in [\![1,p]\!]$, using the coding of $T$ by its \L ukasiewicz path (Proposition \ref{propcodingluka}), we get that 
$\mathbb{E} [\sum_{j_1 \neq \cdots \neq  j_r} \mathbf{1}_{j_1} \cdots \mathbf{1}_{j_r}]$ is equal to
\[\sum_{j_1 \neq \cdots \neq  j_r} \mathbb{P} \left(W_{n}=-1, W_i \geq 0 \, \forall i \in [\![0,n-1]\!] ,\bigl|\{i : X_{i}=-1\}\bigr|=k_n , \widehat{X}_{j_1}= \cdots = \widehat{X}_{j_r} = d \right) \]
Applying Proposition \ref{classiquedecomposition} with \[f(x_1,\ldots,x_{n-k_n}) = \sum_{ j_1 \neq \cdots \neq  j_r \in [\![1,n-k_n]\!]}\1_{ x_{j_1}= \cdots = x_{j_r} = d}, \] we obtain that the latter probability equals
\begin{align*}
     \frac{1}{n} \binom{n}{k_n}\mu(0)^{k_n} (1-\mu(0))^{n-k_n}\mathbb{E}\left[ f\left((Y_i)_{1\leq i \leq n-k_n} \right)\1_{W'_{n-k_n}=k_n-1} \right].
\end{align*}
Moreover, we have by independence \[\mathbb{E}\!\left( f\left((Y_i)_{1\leq i \leq n-k_n} \right)\1_{W'_{n-k_n}=k_n-1} \right) \!= (n-k_n)\cdots(n-k_n-r+1)\,\mathbb{P}\left( Y_1=\! \cdots  \!=Y_r = d ,W'_{n-k_n} \! =k_n-1\right),\]
and by the definition of $W'$ and $(Y_i)_{1\leq i \leq n-k_n}$, we obtain 
\begin{align*}
  \mathbb{P}\!\left(Y_1=\! \cdots  \!=Y_r = d, W'_{n-k_n}\!=k_n\!-\!1\right) 
 & = \mathbb{P}\left( Y_1=\! \cdots  \!=Y_r = d \right)\mathbb{P}\left(W'_{n-k_n}=k_n-1 \mid  Y_1=\! \cdots  \!=Y_r = d \right) \\ 
 & = \left(\frac{\mu(d+1)}{1-\mu(0)}\right)^r\mathbb{P}\left(W'_{n-k_n-r}=k_n-rd-1 \right).
\end{align*} 
Hence applying Corollary \ref{propTW}, we deduce that for every $r \in [\![1,p]\!]$, 
\begin{align} \label{eqesp_r}
    \frac{\mathbb{E} \left[\sum_{j_1 \neq \cdots \neq  j_r} \mathbf{1}_{j_1} \cdots \mathbf{1}_{j_r} \right]}{\mathbb{P}\left( T \in \mathbb{T}_{n}^{k_n} \right)} = (n-k_n)\cdots(n-k_n-r+1) \left(\frac{\mu(d+1)}{1-\mu(0)}\right)^r \frac{\mathbb{P}\left(W'_{n-k_n-r}=k_n-rd-1 \right)}{\mathbb{P}\left(W'_{n-k_n}=k_n-1 \right)},
\end{align}
and combining with Equation \eqref{eqesp_r}, Corollary \ref{corKM}(i) and Corollary \ref{corKM}.(iii), we obtain
\begin{align*}
     \frac{\mathbb{E} \left[\sum_{j_1 \neq \cdots \neq  j_r} \mathbf{1}_{j_1} \cdots \mathbf{1}_{j_r} \right]}{\mathbb{P}\left( T \in \mathbb{T}_{n}^{k_n} \right)}  \underset{n \to \infty}{\sim}  \left(\frac{\mu(1)^{d-1}\mu(d+1)}{\mu(2)^d}\right)^r \frac{k_n^{rd}}{n^{r(d-1)}}.
\end{align*}
Consequently, by Equations \eqref{formespcond} and \eqref{forumuleavecS}, we get that 
\begin{align*}
\mathbb{E}\left[\left(\frac{n^{d-1}}{k_n^d}\phi_{d+1}\left(T_{n}^{k_n} \right) \right)^p \,\right] \underset{n \to \infty}{\sim} \sum_{r=1}^p S(p,r) \! \left(\frac{\mu(1)^{d-1}\mu(d+1)}{\mu(2)^d}\right)^r \! \! \left(\frac{k_n^{d}}{n^{(d-1)}}\right)^{r-p} \! \! \! \! \underset{n \to \infty}{\sim} \left(\frac{\mu(1)^{d-1}\mu(d+1)}{\mu(2)^d}\right)^p \! \!,
\end{align*}
where the final asymptotic follows from the assumption that $k_n^d/n^{d-1} \rightarrow + \infty$ as $n\rightarrow +\infty$ and the identity $S(p,p) =1$. This concludes the proof of the lemma.
\end{proof}

\bibliographystyle{alpha}
\bibliography{sample}

\end{document}